\documentclass[11pt]{article}
\usepackage{amsmath,amsthm,amssymb,cite,enumerate}
\usepackage[a4paper,left=0.5in,right=1in]{geometry}
\usepackage[usenames]{color}
\usepackage{graphicx}
\usepackage{epsfig,stix}
\usepackage{dcolumn}% Align table columns on decimal point
\usepackage{bm}% bold math
\usepackage{authblk} %for authors with multiple affiliations

\begin{document}

\def \d {{\rm d}}

\def \bm #1 {\mbox{\boldmath{$m_{(#1)}$}}}
\newcommand{\mf}[1]{{\mathfrak #1}}

\def \bg {\mbox{\boldmath{$g$}}}
\def \bV {\mbox{\boldmath{$V$}}}
\def \bT {\mbox{\boldmath{$T$}}}
\def \bk {\mbox{\boldmath{$k$}}}
\def \bl {\mbox{\boldmath{$\ell$}}}
\def \bn {\mbox{\boldmath{$n$}}}
\def \bbm {\mbox{\boldmath{$m$}}}
\def \tbbm {\mbox{\boldmath{$\bar m$}}}
\def \g{\mathfrak g}
\def \T {\bigtriangleup}
\newcommand{\msub}[2]{m^{(#1)}_{#2}}
\newcommand{\msup}[2]{m_{(#1)}^{#2}}
\def\Riem{\mathrm{Riem}}
\def\Ric{\mathrm{Ric}}
\def\span{\mathrm{span}}
\def\sl{\mathfrak{s}\mathfrak{l}(2,\mathbb{R})}

\newcommand{\be}{\begin{equation}}
\newcommand{\ee}{\end{equation}}

\newcommand{\beq}{\begin{eqnarray}}
\newcommand{\eeq}{\end{eqnarray}}
\newcommand{\pa}{\partial}
\newcommand{\pp}{{\it pp\,}-}
\newcommand{\ba}{\begin{array}}
\newcommand{\ea}{\end{array}}

\newcommand{\M}[3] {{\stackrel{#1}{M}}_{{#2}{#3}}}
\newcommand{\m}[3] {{\stackrel{\hspace{.3cm}#1}{m}}_{\!{#2}{#3}}\,}

\newcommand{\tr}{\textcolor{red}}
\newcommand{\tb}{\textcolor{blue}}
\newcommand{\tg}{\textcolor{green}}

\def\b{{\kappa_0}}

\def\E{{\cal E}}
\def\B{{\cal B}}
\def\R{{\cal R}}
\def\N{{\cal N}}
\def\L{{\cal L}}

\def\e{e}
\def\bb{b}

\newtheorem{theorem}{Theorem}[section] 
\newtheorem{cor}[theorem]{Corollary} 
\newtheorem{lemma}[theorem]{Lemma} 
\newtheorem{prop}[theorem]{Proposition}
\newtheorem{con}[theorem]{Conjecture} 
\newtheorem{definition}[theorem]{Definition}
\newtheorem{remark}[theorem]{Remark}  
\newtheorem{q}[theorem]{Question}

\title{Left-invariant Pseudo-Riemannian metrics on Lie groups: \\
The null cone}

\author[$\clubsuit$]{Sigbj\o rn Hervik}

\affil[$\clubsuit$]{Dept. of Mathematics and Physics, University of Stavanger,  N-4036 Stavanger, Norway \newline Email: sigbjorn.hervik@uis.no}

\maketitle

\abstract{We study left-invariant pseudo-Riemannian metrics on Lie groups using the moving bracket approach of the corresponding Lie algebra. We focus on metrics where the Lie algebra is in the null cone of the $G=O(p,q)$-action; i.e., Lie algebras $\mu$ where zero is in the closure of the orbits: $0\in\overline{G\cdot \mu}$. We provide examples of such Lie groups in various signatures and give some general results. For signatures $(1,q)$ and $(2,q)$ we classify all cases belonging to the null cone. More generally, we show that all nilpotent and completely solvable Lie algebras are in the null cone of some $O(p,q)$ action. In addition, several examples of non-trivial Levi-decomposable Lie algebras in the null cone are given. }

\vspace{.2cm}
\noindent

%\tableofcontents
\section{Introduction} 
In this paper we will investigate Lie groups equipped with a left-invariant pseudo-Riemannian metric. Such spaces can be viewed as homogeneous spaces in the sense that the isometry group acts transitively, and the Lie group acts simply transitive on itself. Due to the correspondence between the Lie group and its Lie algebra, the study reduces to considering pseudo-Riemannian metrics on its Lie algebra. 

There are many studies that has considered such metics before, especially in the Riemannian case \cite{Milnor:76,DM,Wolter,Wolter2, Heber,GHG,L1,L2,L3,Sig,L4,Lauret06, Lauret09, Lauret10, LW11,L5}. There have also been some studies for the Lorentzian and the more general pseudo-Riemannian case \cite{CF, solwaves,CR19a,CR19b,CR21,FFS,Markestad}. Most of these studies have been concerned with Einstein metrics or other distinguished metrics on solv- and nilmanifolds. In Markestad's Master thesis \cite{Markestad} the focus was on the degenerate structure of pseudo-Riemannian metrics with vanishing curvature invariants (VSI metrics)\cite{Hervik12,Hervik14}.  A VSI metric is defined as follows: Consider the set of all the curvature tensors the Riemann tensor and its covariant derivatives, $\{{\rm Riem}, \nabla{\rm Riem}, ..., \nabla^{(k)}{\rm Riem},...\}$ . A polynomial curvature invariant, $I$, is an invariant scalar being a polynomial of the (components) of the curvature tensors. If, for all such invariants, $I=0$, then we call the metric VSI. 

In the Lorentzian case such metrics have been studied and all classified, in particular, they are all Kundt metrics \cite{VSI,Higher, CFHP, Kundt}. From a physics point of view these metrics are interesting because they correspond to gravitational wave-like solutions of Einstein theory as well as many other modified theories of gravity \cite{HorowitzSteif,CGHP}. The general pseudo-Riemannian case these VSI metrics are more elusive and lacks a complete understanding partly because the class is richer in more general signatures \cite{Hervik14}.  In this paper we will consider systematically the subclass of these VSI metrics which can be viewed as a Lie group endowed with a left-invariant metric and where its Lie algebra is in the null cone under the action of the orthogonal groups $O(p,q)$. 

Essential for this approach is considering the orbits of the group action $G\cdot\mu$, where $\mu$ is the Lie bracket (defined later). Two Lie algebras are isomorphic (over $\mathbb{R}$) if they are in the same orbit under the $G=GL(n,\mathbb{R})$-action. The corresponding pseudo-Riemannian left-invariant metrics are isometric if they are in the same $G=O(p,q)$-action. The $O(p,q)$-null cone, $\mathcal{N}$, is defined as those $O(p,q)$-orbits having zero in their closure. Even if all $GL(n,\mathbb{R})$-orbits contain zero in its closure, not all $O(p,q)$-orbits do. Here we will investigate conditions for a Lie algebra being in the $O(p,q)$-null cone. 

Aside from listing a lot of examples, including several Ricci-flat pseudo-Riemannian nil- and solvmanifolds, we also give some general results. Two important corollaries are:
\begin{enumerate}
\item{\bf Corollary \ref{CorSimple}: }{}Let the Lie algebra $\mu$ be semi-simple. Then $\mu$ is not in the null cone of any $O(p,q)$-action. 
\item{\bf Corollary \ref{CorNil}: }{} Assume that a $(p+q)$-dimensional Lie algebra $\g$ is in the null cone under an $O(p,q)$-action where $p\leq q$. Then  $\g$ has a nilpotent subalgebra of dimension $q$. 
\end{enumerate}
Even if the semi-simple Lie algebras cannot be in the nullcones, we have found examples of Lie algebras with non-trivial Levi decompositions in the null cone. For example, in dimension six, $\mathfrak{s}\mathfrak{l}(2,\mathbb{R})\oplusrhrim \mathfrak{n}_{3,1}$ is in the null-cone of the $O(2,4)$ action and $\mathfrak{s}\mathfrak{l}(2,\mathbb{R})\oplusrhrim \mathfrak{s}_{3,1, a=1}$ is in the null cone of the $O(3,3)$ action. Other examples are also given. 

There are also solvable examples which are \emph{not} in the null cone of any $O(p,q)$ action: simplest examples here are the three-dimensional Lie algebras of the Euclidean and hyperbolic motion in the plane. On the other hand, if the Lie algebra is nilpotent, or completely solvable, then it is in the null cone of $O(p,p)$ or $O(p,p+1)$ action, see Propositions \ref{PropNil} and \ref{PropCSol}. 

Before we embark on our results, we will first review some of the mathematical background.

\subsection{The moving bracket approach}

Consider a vector space $V$. Then a Lie algebra $\g\cong V$ (as a vector space) can be defined using the Lie algebra bracket $\mu$: $[X,Y]=\mu(X,Y)$, for all $X,Y\in V.$ As such we consider $\mu\in \mathcal{V}:=V\otimes\wedge^2V^*$. We will also equip $V$ with a (fixed)  pseudo-Riemannian metric $g$ of signature $(p,q)$ where $\dim(V)=p+q$ in the folloing way: Assume $p\leq q$ so that $q-p=k$,  we will choose a null basis $\{e_1,...,e_{2p+k}\}$ with a corresponding co-basis $\{\theta^1,...,\theta^{2p+k}\}$: 
\beq \label{metric}
g=2\theta^1\odot\theta^2+...+2\theta^{2p-1}\odot\theta^{2p}+\sum_{i=2p+1}^{2p+k}\theta^i\odot\theta^i,
\eeq
where $\odot$ denotes the symmetric tensor product. Henceforth, we will omit the $\odot$ when we considering symmetric tensor objects.

A common choice is also an orthonormal frame, however, for our purposes a null-frame like above is more convenient due to the boost-weight decomposition, see later section  \ref{boostweightdecomp}.
 
The structure group of the frame bundle for the pseudo-Riemannian metric is the group $G=O(p,q)$.  The structure group acts naturally on the bracket: Given $A\in O(p,q)$, then
\[ (A\cdot\mu)(X,Y)=A^{-1}(\mu(AX,AY)), \]
Hence, for a given $\mu$ we can consider the orbit $G\mu\subset \mathcal{V}$, 

We will henceforth always assume that $p+q\geq 3$ so that $O(p,q)$ is semi-simple. 

%\subsection{The connection and the curvature tensors} 

Given a Lie algebra bracket $\mu$, we can now define a left-invariant torsion-free connection, $\nabla$, by requiring
\beq\label{torsionfree}
\nabla_{X}Y-\nabla_YX=\mu(X,Y), \quad \forall X,Y\in V.  
\eeq
Furthermore, we can define the (Riemann) curvature operator and its covariant derivatives.  

The Riemann curvature tensor, $\Riem$,   will be a homogeneous polynomial of order 2 in $\mu$. In particular, this means that viewing the Riemann tensor as a function of $\mu$, $\Riem(\mu)$, a scaling of $\mu$ with $c\in\mathbb{R}$ gives: 
\[ \Riem(c\mu)=c^2\Riem(\mu).\] 
For the covariant derivatives:
\[ \nabla^{(K)}\Riem(c\mu)=c^{K+2} \nabla^{(K)}\Riem(\mu)\quad \Leftrightarrow\quad \nabla^{(K)}\Riem\text{ is homogeneous polynomial of order }(K+2) \text{ in }\mu. \] 
This structure will be utilised later in the classification and on the properties of the $\Riem$ and $\nabla^{(K)}\Riem$ tensors. 

It is important to note that since the group $O(p,q)$ is the structure group of the frame bundle, each point in its orbit represents equivalent pseudo-Riemannian structures. Polynomial curvature invariants are $O(p,q)-$invariant polynomials of the components of the curvature tensors, and are thus constants along orbits. They are polynomial in the curvature tensors and thus polynomial in the structure constants. 

With the reference to a left-invariant null-basis $\{e_a\}$ and its corresponding left-invariant co-basis $\{\theta^a\}$, we introduce the structure coefficients so that $[e_a,e_b]=C^c_{~ab}e_c$ (summation is assumed over Latin indices, early in the alphabet). These represent the components of $\mu$ with respect to the null-frame. 
In the book \cite{SW} all lower-dimensional Lie algebras  are listed. Whenever we refer to a specific Lie algebra we will consistently use the enumeration in \cite{SW}. For example, the Lie algebra $\mathfrak{s}_{6,242}$ is the six-dimensional solvable Lie algebra given on page 296 in \cite{SW}.  

\section{The null cone}
Under the orbit of $G=O(p,q)$ acting on $\mathcal{V}$ we will define the null cone, $\N\subset\mathcal{V}$ as: 
\[ \N:=\left\{\text{Lie algebra }\mu\in \mathcal{V}: 0\in\overline{G\mu}\right\}; \]
i.e., the null cone consists of all orbits $G\mu$ having $0$ in their closure, $\overline{G\mu}$. 

Let us point out a few facts \cite{RS}: 
\begin{theorem} \label{thmRS}
Given the null cone as defined above. Then: 
\begin{enumerate}
\item{} The orbit $\{0\}$ is the unique closed orbit in $\N$. 
\item{} Consider all polynomial invariants constructed from the bracket under the action of $O(p,q)$. Then given a Lie algebra $\mu\in\mathcal{V}$, these are all zero if and only if $\mu\in\N$. 
\item{} Given a $\mu\neq 0$, and a Cartan decomposition $\g ={\mathfrak k}\oplus{\mathfrak p}$ of $\mathfrak{o}(p,q)$, then there exists an $X\in \mathfrak{p}$ so that $\lim_{t\rightarrow \infty} e^{tX}\cdot\mu=0.$
\end{enumerate}
\end{theorem}
For any non-zero $\mu$ we thus have an $X\in \mathfrak p$ which generates a contraction to the origin. This immediately implies that
{\it all $O(p,q)$-invariant polynomials of $\mu$ must vanish on $\N$. } Of special interest are curvature invariants which are necessarily all  zero as well, i.e., they are VSI spaces \cite{Hervik12}. These invariants are constructed from the Riemann curvature tensor and its derivatives.

Also of interest is the Killing form, which is a homogeneous polynomial of order 2. We define the Killing operator ${K}\in{\rm End}(V)$, by:
\[ B(v,w)=g({K}(v),w), \quad \forall v,w\in V, \]
where $B$ is the Killing form. The invariants of the Killing form, as well as the eigenvalues, need to be zero if $\mu$ is in the null cone, hence: 
\begin{prop}
Let $\mu\in{\mathcal{N}}\subset V\otimes\wedge^2V^*$ be in the null cone. Then $K$ is nilpotent and:
\begin{enumerate}
\item{} if $k=0$, then the index of nilpotency of $K$ is $\leq 2p$. 
\item{} if $k\geq 1$, then the index of nilpotency of $K$ is $\leq(2p+1)$. 
\end{enumerate}
\end{prop}
\begin{proof}
If $k=0$ this trivially holds since $\dim V=2p$. For $k\geq 1$, we note first that $K_{A\cdot\mu}=A^{-1}K_{\mu}A$ for $A\in O(p,q)$. 
Using that all eigenvalues are zero, this can be seen by using the Jordan canonical forms combined with Petrovs canonical forms for arbitrary signatures \cite{Petrov}.
\end{proof}
We point out that this gives an upper bound of the level of nilpotency of the operator $K$, however, in practice the level is lower as we will see later. Notwithstanding this, it immediately implies: 
\begin{cor}\label{CorSimple}
Let $\mu$ be semi-simple. Then $\mu$ is not in the null cone of any $O(p,q)$-action. 
\end{cor}

\subsection{The boost weight decomposition of tensorial objects} 
\label{boostweightdecomp}
By using the maximal Abelian subalgebra $\mathfrak{a}\subset \mathfrak{p}$ of dimension equal to the real rank of $O(p,q)$, we define a set of positive weights ${\sf b}_i\in\mathfrak{a}$, $1\leq i \leq \min(p,q)$. Henceforth, we will assume that $p\leq q$, so that $q=p+k$, for some $k\in \mathbb{N}_0$. Using an appropriate scaling, we can assume that the eigenvalues of each ${\sf b}_i$ acting on $V$ are $0, \pm 1$. As each ${\sf b}_i$ is in $\mathfrak{o}(p,q)$, the action of the ${\sf b}_i$ can now be lifted tensorially so that it acts on any tensor-space $\bigotimes_nV\otimes\bigotimes_m V^*$. Since each ${\sf b}_i$ commute, we can for any tensor, $T$, perform an eigenvalue decomposition \cite{Hervik10,Hervik12}:
\beq\label{bwdecomp} T=\sum_{(b_1,b_2,...,b_p)\in \Delta}(T)_{(b_1,b_2,...,b_p)}, \eeq
where $\Delta$ is a finite subset of $\mathbb{Z}^p$. With the notation $(T)_{(b_1,b_2,...,b_p)}$ we mean the eigenvalue component of $T$ so that ${\sf b}_i\cdot(T)_{(b_1,b_2,...,b_p)}=b_i(T)_{(b_1,b_2,...,b_p)}$, $\forall i$. We will refer to the $p$-tuple $(b_1,b_2,...,b_p)$ as the \emph{boost weight}. 

We note that the metric tensor is invariant under the group $O(p,q)$, hence $g=(g)_{(0,..,0)}$. Furthermore
\[ (T\otimes S)_{(b_1,...,b_p)}=\sum_{(b_1,...,b_p)=(\hat{b}_1,...,\hat{b}_p)+(\tilde{b}_1,...,\tilde{b}_p)}(T)_{(\hat{b}_1,...,\hat{b}_p)}\otimes(S)_{(\tilde{b}_1,...,\tilde{b}_p)}. \] 
The boost weight therefore adds vectorially when we take tensor products. 

The subset $\Delta$ will for the different curvature tensors (for the split case there are additional conditions): 
\begin{enumerate}
\item{} Riemann tensor: $\sum_i |b_i |\leq 4$ and $|b_i|\leq 2$. 
\item{} Ricci tensor: $\sum_i |b_i| \leq 2$.
\end{enumerate}
In the split cases, $p=q$, then in addition, $\sum b_i$ is even for both the Riemann and Ricci tensors (or for any other even-indexed tensor). For an odd-indexed object, like the structure constants, then $\sum b_i$ must be odd. 

\paragraph{Permutations:} We note that the metric (\ref{metric}) has permutation symmetries (relabelling the 1-forms): 
\[ \sigma: \qquad \theta^{2i-1} \rightleftharpoons \theta^{2i}, \qquad (\theta^{2i-1},\theta^{2i})\rightleftharpoons (\theta^{2j-1},\theta^{2j}), \qquad 1\leq i,j\leq p. \] 
These permutations will produce multiple sets in $\mathcal{V}$ being identical as pseudo-Riemannian spaces, thus it is useful to consider the quotient $\mathcal{V}/ \sigma$. All the cases below can be written as an inequality 
\[ x_1b_i+x_2b_2+...+x_pb_p\leq -1,\] 
and we will choose the representative of $\mathcal{V}/ \sigma$ where the coefficients are non-negative and (reversely) well-ordered: 
\[ x_1\geq x_2\geq ...\geq x_p\geq 0.\] 

These coefficients are directly related to the element  $X\in \mathfrak{p}$ in Theorem \ref{thmRS}. More specifically, using the metric (\ref{metric}), then the action of $X\in \mathfrak{p}$ on the basis 1-forms gives: 
\[ e^{tX}\left\{\theta^1,\theta^2,\theta^3,\theta^4,...,\theta^{2p-1},\theta^{2p},\theta^i\right\}=\left\{e^{-tx_1}\theta^1,e^{tx_1}\theta^2,e^{-tx_2}\theta^3,e^{tx_2}\theta^4,...,e^{-tx_p}\theta^{2p-1},e^{tx_p}\theta^{2p},\theta^i\right\}\] 
This means that for all tensor components, $(T)_{(b_1,...,b_p)}$, obeying $x_1b_i+x_2b_2+...+x_pb_p\leq -1$, will then under the action:
\[ \lim_{t\rightarrow\infty}e^{tX}(T)_{(b_1,...,b_p)}=\lim_{t\rightarrow\infty}e^{t(x_1b_i+x_2b_2+...+x_pb_p)}(T)_{(b_1,...,b_p)}=0;\] 
hence, 0 will be in the closure, and $T$ be in the null cone, as claimed. 
\subsection{Classification of orbits in the null cone}
The classification of orbits in the null cone will be a two-step process. First we will classify orbits in the more general null cone of any tensor in $\mathcal{V}$: 
\[ \N_{\mathcal V}:=\left\{T\in \mathcal{V}: 0\in\overline{G\cdot T}\right\}. \]
The first step is to find canonical representatives in the orbits in $\N_{\mathcal V}$. 
However, not all tensors in $\mathcal{N}_{\mathcal V}$ represent a Lie algebra. The set of tensors representing a Lie algebra is a subvariety of $\mathcal{V}$: 
\[ \mathcal{J}:=\left\{T\in \mathcal{V}: T(X,T(Y,Z))+T(Y,T(Z,X))+T(Z,T(X,Y))=0, ~\forall X, Y, Z\in V \right\}; \] 
i.e., those that satisfies the Jacobi identity. Clearly, 
\[ \N=\N_{\mathcal V}\cap\mathcal{J}.\] 
Hence, the second step is to restrict to those that obey the Jacobi identity. Some of the cases will turn out to be not realisable unless you restrict to a more special case. In particular, in signatures $(1,1+k)$ and $(2,2+k)$ we will give a full classification of realisable cases in the null cone. For other signatures, partial results are provided. 

\section{Lorentzian case $O(1,n-1)$.}
The Lorentz group is of real rank 1 so there is only one ${\sf b}\equiv{\sf b}_1$. There are two cases to consider according to the boost weight. Let us introduce a null-frame $\{e_1,e_2,e_i\}$, where $3\leq i \leq n$, aligned with ${\sf b}$ so that ${\sf b}(e_1)=e_1$, ${\sf b}(e_2)=-e_2$, and ${\sf b}(e_i)=0$. This implies that  the bracket can be decomposed as: 
\[ \mu=(\mu)_{-1}+(\mu)_{-2}.\]
In the Lorentizan case tensors with such a decomposition are referred to as \emph{type III} tensors. If, in addition, $(\mu)_{-1}=0$ so that $ \mu=(\mu)_{-2}$, then it is of the more special case of \emph{type N}. 

\subsection{$\mu$ of type III} 
\label{sect:typeIII}
Here, $\mu=(\mu)_{-1}+(\mu)_{-2}$. This allows for the only non-zero components: 
\[ \text{b.w.} -2:~ C^2_{~1i}, \qquad\text{b.w.} -1:~ C^2_{~12},~C^2_{~ij}, ~C^i_{~1j}.\] 
Thus, $[\g,\g]:=\g_1\subset \text{span}\{e_2,e_i\}$. Furthermore, $[\g_1,\g_1]\subset \text{span}\{e_2\}$ and $[\g,\g_1]\subset \text{span}\{e_2,e_i\}$. Hence, in general this is a solvable Lie algebra with a 2-step nilpotent nilradical. 

The Jacobi identity implies the following non-trivial equations among these structure coefficients: 
\beq\label{JtypeIII}
 C^2_{~12}C^2_{~ij}-\sum_{l=3}^{n}\left(C^2_{~il}C^l_{~1j}-C^2_{~jl}C^l_{~1i}\right)=0
 \eeq

For this case $\Riem$, $\Ric$ and the Killing form have only boost weight $-2$ components which implies that these tensors quadratic in $\mu$ are all of type N. Explicitly, the only possible non-zero components (independent) for the Ricci tensor and Killing form are:
\beq
%g(e_1,\Riem(e_1,e_j)e_i)&=&...\\
\Ric_{11} &=&\frac 12\sum_{i<j}\left(C^2_{~ij}\right)^2-\frac 12 \sum_{i,j}\left[\left(C^i_{~1j}\right)^2+C^i_{~1j}C^j_{~1i}\right]+C^2_{~12}\sum_{i}C^i_{~1i}\\
B(e_1,e_1)& = & \left(C^2_{~12}\right)^2+\sum_{i,j}C^i_{~1j}C^j_{~1i}.
\eeq
Similarly, higher derivatives can also be computed and we note that $\nabla^{(k)}\Riem$ has only boost-weight $-(2+k)$ components. 

As a class of examples of these Lie algebras, consider dimension 4. Then the Jacobi identity, eq.(\ref{JtypeIII}), reduces to 
\[ C^2_{~34}(C^2_{~12}-C^3_{~13}-C^4_{~14})=0.\] 
If we assume $C^2_{~34}=c\neq 0$, then this example allows for the 2-step nilradical $\mathfrak{n}_{3,1}$. Then, defining: 
\[ C^2_{~12}=a+b, \quad C^3_{~13}=a, \quad C^4_{~14}=b,\] 
the Jacobi identity is satisfied. Setting the components: 
\[ C^3_{~14}=\alpha, \quad C^4_{~13}=\beta, \] 
gives 
\[ \Ric=-\frac 12\left[(\alpha+\beta)^2-c^2-4ab\right]\theta^1\theta^1, \qquad B=2\left(a^2+ab+b^2+\alpha\beta\right)\theta^1\theta^1.\] 
This class of metrics allows for all the solvable Lie algebras $\mathfrak{s}_{4,6}-\mathfrak{s}_{4,11}$. Furthermore, it is also clear that there is a subclass which is Ricci-flat. 

Consider a concrete example, namely $\mathfrak{s}_{4,8}$, where $a=1$, $c=1$, $\alpha=\beta=0$. Then, an explicit left-invariant metric can be given: 
\[ \theta^1 = \d x\ ~~ \theta^2=\d y - z\d w + (bzw - (1 + b)y) \d x, ~~ \theta^3=\d z- z\d x, ~~\theta^4 =\d w - bw\d x,\]
so that 
\[ g=2\theta^1\theta^2+\left(\theta^3\right)^2+\left(\theta^4\right)^2.\]  
If $b=-1/4$ then this is a Ricci-flat metric.

\subsection{$\mu$ of type N} Here, $\mu=(\mu)_{-2}$. This gives the only non-zero components
\[ \text{b.w.} -2:~C^2_{~1i}.\] 
Hence, $\g$ is 2-step nilpotent and the Riemann tensor is zero, and consequently, the metric is flat. 

The only admissable non-abelian Lie algebra is $\mathfrak{n}_{3,1}\oplus \mathbb{R}^{k-1}$. By a spatial rotation we can assume that the only non-zero component is $C^2_{~13}=a$. Thus, all left-invariant metrics of this type can be written: 
\[ g=2\d x(\d y-ax\d z^1)+\sum_{i=1}^{k}(dz^i)^2.\]

\section{Signature $(2,2+k)$.} 
The groups $O(2,2+k)$ have real rank 2, so here b.w. decomposition is given in terms of the two eigenvalues of the positive weights. Any eigentensor can therefore be classified using a pair of integers $(b_1,b_2)\in\mathbb{Z}^2$. For a tensor  $\mu\in V\otimes\wedge^2V^*$, we can plot the possible non-zero eigenspaces on the lattice $\mathbb{Z}^2\subset \mathbb{R}^2$. The null cone can now be described by the union of the following two cases: 
\begin{enumerate}
\item{Case $[3,1]$:} This type is characterised by allowing only $C^c_{~ab}$ whose boost weights $(b_1,b_2)$ obey $3b_1+b_2<0$, all other are zero, see fig. A tighter bound can be given by $3b_1+b_2\leq -1$. 
\item{Case $[3,2]$:} This type is characterised by allowing only $C^c_{~ab}$ whose boost weights $(b_1,b_2)$ obey $3b_1+2b_2<0$, all other are zero, see fig. A tighter bound can be given by $3b_1+2b_2\leq -1$. 
\end{enumerate}
Note that the intersection of these two types can be described with $2b_1+b_2\leq -1$ (case $[2,1]$ below). 

These types have numerous subcases. In order to classify them we will use the following notation. If the case can be described by non-zero structure coefficients only if the boost weight obeys (normalising right hand side to $-1$): 
\[ xb_1+yb_2\leq -1 \] 
then this case will be denoted as case $[x,y]$, where, $x$ and $y$ would  be rationals, $x,y\in\mathbb{Q}$. 

\begin{center}
\includegraphics[width=15cm]{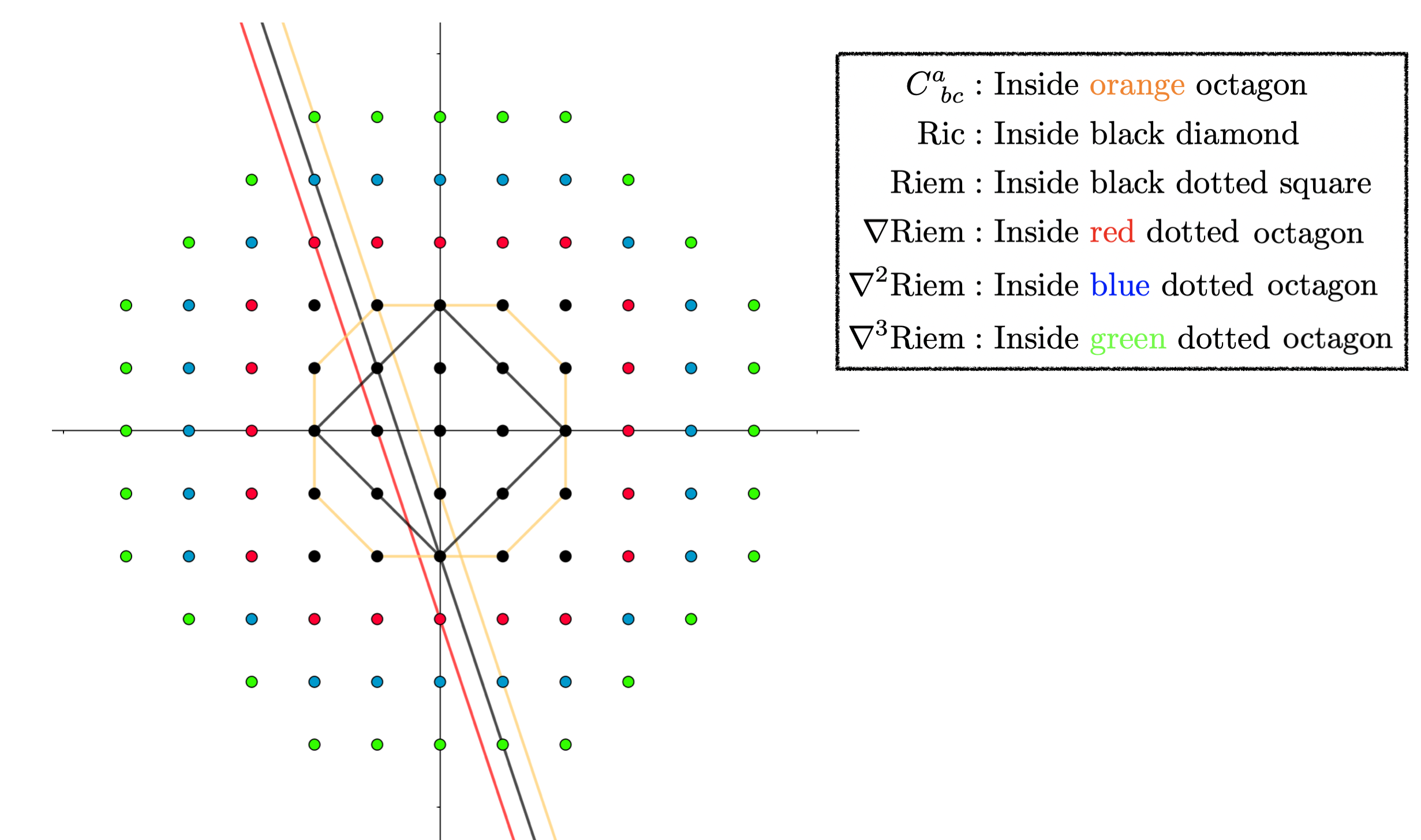}\\
{The figure illustrates the allowable boost weights for the structure constants and some curvature tensors. The black diamond  (the Ricci diamond) corresponds to any symmetric 2-tensor and thus includes also the Killing form. The straight orange line is the example of case [3,1] below corresponding to boost weights obeying $3b_1+b_2\leq -1$. The black and red lines correspond to the Riemann and the $\nabla$Riemann tensors, respectively, for case [3,1]. }
\end{center}

\subsection{Classification of cases}
Assume a left-invariant frame: $\{e_1,e_2,e_3,e_4,e_i\}$ with its left-invariant 1-forms: $\{\theta^1,\theta^2,\theta^3,\theta^4,\theta^i\}$. 

The pseudo-Riemannian left-invariant metric is taken to be: 
\[ g=2\theta^1\theta^2+2\theta^3\theta^4+\sum^{4+k}_{i=5}(\theta^i)^2.\]
For all the cases below, we have given the: 
\begin{itemize}
\item{} {\textsl{Universality order}}: This is the lowest number $k$ so that a polynomial of order $k$ in $C^a_{~bc}$ does not have boost-weights inside the Ricci-diamond. This implies that any symmetric 2-tensor being homogeneous of order $k$ in $C^a_{~bc}$ is zero. 
\item{} {\textsl{Jordan type}}: The structure of the Jordan block form of the Ricci and Killing operators. Only non-zero blocks are given and all the blocks are nilpotent for the structure constants in the null cone; e.g. type ${\mathbf 3}\oplus{\mathbf 2}$ implies that the Jordan block form consists of a $3\times 3$ nilpotent block and a $2\times 2$ nilpotent block, the rest is zero. 
\end{itemize}

In appendix \ref{Cabc} the relevant structure constants for the null cone are listed, and it can be assumed that all others are zero. They are listed according to the boost weights and for each case below some of the boost weights are disallowed. 
\subsubsection{Case [3,1]} 
Universality order: 7. Ricci and Killing operators of Jordan type: ${\mathbf 4}$.

Considering the allowed boost-weight components the Jacobi identity implies (among others): 
\[  C^3_{~14}(C^1_{~13}-C^4_{~34})=0.\] 
If $C^3_{~14}=0$ then this is actually a simpler case (case [2,1] below). Assuming $C^3_{~14}\neq 0$ implies $C^1_{~13}=C^4_{~34}$. 
This case  can be split  into two whether $C^1_{~13}=C^4_{~34}$ equals zero or not, corresponding to whether the Lie algebra is solvable or not.  

\paragraph{ Case [3,1]a: Non-solvable case.}  This case allows for algebras with semi-simple subalgebras (but it cannot be semi-simple itself). The simplest example seems to be in neutral 4 dimensions ($k=0$): $\mathfrak{s}\mathfrak{l}(2,\mathbb{R})\oplus \mathbb{R}$, with structure constants: 
\[ C^1_{~13}=C^4_{~34}=2, \qquad C^3_{~14}=1.\] 
Riemann tensor is non-zero, and Ricci tensor  and Killing form are: 
\[  \Ric=6(\theta^4\theta^1-\theta^3\theta^3), \qquad B=-8(\theta^1\theta^4-\theta^3\theta^3).\]
For these examples, both of the Ricci tensor and the Killing form are of Jordan type ${\mathbf 4}$. 

Hence, this case can be realised in any dimension of signature $(2,2+k)$. 

\paragraph{ Case [3,1]b: Solvable case.} If $C^1_{~31}=C^4_{~34}=0$, then we get the following derived series: 
\[ [\g,\g]:=\g_1\subset{\rm span}\{e_2,e_3,e_4,e_i\}, \quad  [\g_1,\g_1]:=\g_2\subset{\rm span}\{e_2,e_4,e_i\}, \] 
\[ [\g_2,\g_2]:=\g_3\subset{\rm span}\{e_2,e_4\},\quad 
[\g_3,\g_3]=0.\] 
and hence, it is solvable. 

As an example, we can choose the solvable Lie algebra $\mathfrak{s}_{4,1}$ and the structure constants: 
\[ C^4_{~12}=1, ~~C^3_{~13}=1,~~C^3_{~14}=1.\]
The Riemann tensor is non-zero, and the Ricci tensor and Killing form are: 
\[ \Ric=-\tfrac 12\theta^1\theta^1+2\theta^1\theta^3+\theta^1\theta^4-\tfrac 12\theta^3\theta^3, \qquad B=\theta^1\theta^1\] 
Ricci tensor is of Jordan type ${\mathbf 4}$, and the Killing form of type ${\mathbf 2}$.

\subsubsection{Case [3,2]}
Universality order: 7. Ricci and Killing operators of Jordan type: ${\mathbf 3}\oplus{\mathbf 2}$.

Using the Jacobi identity, one gets for the boost weight $(0,-1)$ component: 
\[ C^4_{~23}C^3_{~1i}=0, \quad  C^4_{~23}C^2_{~4i}=0, \quad C^i_{~14}C^4_{~23}=0.\] 

This implies that  $C^4_{~23}=0$ or $C^3_{~1i}=C^2_{~4i}=C^i_{~14}=0$. Hence, the Jacobi identity forces the Lie algebra  to belong to the more special subcases below.

\subsubsection{Case [2,1]}
Universality order: 5. Ricci and Killing operators of Jordan type: ${\mathbf 3}\oplus{\mathbf 2}$.

Also this case can be split in a solvable and non-solvable case. 

\paragraph{ Case [2,1]a: Non-solvable case.} 

 Interestingly, there are examples of the case $[2,1]$ using a non-trivial Levi decomposition in dimension 6.  This time we choose the solvable Levi factor to be the 3-dimensional nilpotent algebra  $\mathfrak{n}_{3,1}$. We let $\mathfrak{s}\mathfrak{l}(2,\mathbb{R})=\span\{e_1,e_3,e_5\}$ and $\mathfrak{n}_{3,1}=\span\{e_2,e_4,e_6\}$, so that: 
\beq\label{sl2Rn}
&&(-1,1):~~ C^3_{~15}=1,~~ C^2_{~46}=a, ~~C^6_{~14}=-b, \nonumber \\
&&(0,-1):~~C^1_{~13}=2, ~~C^5_{~53}=-2, ~~C^4_{~34}=b, ~~C^6_{~36}=-b, ~~C^4_{~56}=b.
\eeq
There are four different Lie algebras according to the values of $a$ and $b$: 
\begin{enumerate}
\item{} $a\neq 0$, $b=1$:  $\mathfrak{s}\mathfrak{l}(2,\mathbb{R})\oplusrhrim \mathfrak{n}_{3,1}$ (non-trivial Levi decomposable).
\item{} $a=0$, $b=1$: $(\mathfrak{s}\mathfrak{l}(2,\mathbb{R})\oplusrhrim \mathbb{R}^2)\oplus\mathbb{R}$
\item{} $a\neq 0$, $b=0$:  $\mathfrak{s}\mathfrak{l}(2,\mathbb{R})\oplus \mathfrak{n}_{3,1}$ (decomposable).
\item{} $a=0$, $b=0$: $\mathfrak{s}\mathfrak{l}(2,\mathbb{R})\oplus \mathbb{R}^3$
\end{enumerate}
The last of these also exists in dimension 5 by restricting to  $\mathfrak{s}\mathfrak{l}(2,\mathbb{R})\oplus \mathbb{R}^2$.
Computing the Ricci and Killing forms: 
\[ \Ric=\left({b^2}+4\right)\theta^1\theta^5-(ab(1-b)-2a-b)\theta^1\theta^6-3\left(\tfrac{b^2}2+2\right)\theta^3\theta^3, \quad B=2(b^2+4)\left(-\theta^1\theta^5+\theta^3\theta^3\right). \] 
For these, the Ricci tensor and Killing form are of Jordan type  ${\mathbf 3}\oplus{\mathbf 2}$.
\paragraph{ Case [2,1]b: Solvable case.} 
A class of solvable metrics in 5-dimensions can also be found here. 
Consider the 5-dimensional solvable Lie algebra: 
\beq
C^4_{~12}=d  ,~~C^1_{~13}=a, ~~C^2_{~23}=b, ~~C^5_{~14}=c,~~C^4_{~34}=-(a+b), ~~C^5_{~35}=-(2a+b)
\eeq
Ricci tensor and Killing form: 
\beq
\Ric= -3 c \left(b +2 a \right) \theta^1\theta^5 -\frac 12\left({3 a^{2}}-{b^{2}}+d^{2}\right) \theta^3\theta^3, \quad 
B=3\left(2 a^{2}+2 a b +b^{2}\right) \theta^3\theta^3
\eeq
Here, the Ricci type is  ${\mathbf 3}\oplus{\mathbf 2}$, while the Killing form has type  ${\mathbf 2}$. Note that there are also simpler Ricci type in this class, in particular, if $b=-2a$ and $d=\pm a$ then this metric is Ricci flat. 

As a concrete example in dimension 6, choose the solvable Lie algebra $\mathfrak{s}_{6,242}$ which is a solvable Lie algebra  having a 3-step nilpotent nilradical. Given the structure coefficients: 
\[ C^2_{~46}=C^4_{~56}=C^6_{~16}=C^4_{~14}=C^5_{~35}=C^4_{~34}=C^2_{~32}=1, \quad C^2_{~12}=2, \] 
give the Ricci tensor and Killing form: 
\[ \Ric=\tfrac52\theta^1\theta^1+3\theta^1\theta^3-3\theta^1\theta^6+\theta^3\theta^3, \quad B=6\theta^1\theta^1+6\theta^1\theta^3+3\theta^3\theta^3. \] 
For this example, the Ricci tensor is of Jordan type  ${\mathbf 3}\oplus{\mathbf 2}$ while the Killing form is of type  ${\mathbf 2}\oplus{\mathbf 2}$.

\subsubsection{Case [2,$\tfrac{\mathbf 1}{\mathbf 2}$]}
Universality order: 5. Ricci and Killing operators of Jordan type: ${\mathbf 3}$.

Using the Jacobi identity we get for the boost-weight $(-1,0)$ component: 
\[ C^3_{~14}C^4_{~3i}=0 \quad \Leftrightarrow \quad C^3_{~14}=0~\vee~C^4_{~3i}=0.\] 
Hence, a Lie algebra needs to be one of the subcases  [$\tfrac{3}{ 2}$,$\tfrac{1}{2}$] or [$1,0$]. 

In any case, computing the derived series: 
\[ [\g,\g]:=\g_1\subset{\rm span}\{e_2,e_3,e_4,e_i\}, \quad  [\g_1,\g_1]:=\g_2\subset{\rm span}\{e_2, e_4\}, \quad  [\g_1,\g_2]=\g_3\subset{\rm span}\{e_2\}\quad  [\g_1,\g_3]=0.\]
So any Lie algebra belonging to any of its subcases is necessarily  solvable, with a 4-step nilpotent nilradical (or simpler). 

\subsubsection{Case [$\tfrac{\mathbf 3}{\mathbf 2}$,$\tfrac{\mathbf 1}{\mathbf 2}$]}
Universality order: 4. Ricci and Killing operators of Jordan type: ${\mathbf 2}\oplus{\mathbf 2}$.

The Lie algebra is solvable, with a 4-step nilpotent nilradical. 
Examples of this case can be found. Consider the 5-dimensional solvable Lie algebra with the nilradical $\mathfrak{n}_{4,1}$  (isomorphic to the 1-parameter family of Lie algebras $\mathfrak{s}_{5,35}$, see below): 
\[                            C^2_{~45} = 1, \quad 
                           C^4_{~35} = 1, \quad 
                           C^4_{~14} = a, \quad 
                           C^5_{~15} = b, \quad 
                         C^3_{~13} = a - b, \quad 
                         C^2_{~12} = a + b, \quad 
                           C^3_{15} = c
\] 
The Riemann tensor is non-zero and the Ricci tensor and Killing form is: 
\[ 
\Ric=\frac b2(8a-3b)\theta^1\theta^1-(c-1)\theta^1\theta^3, \quad B=3(a^2+b^2)\theta^1\theta^1.
\] 
The Ricci operator is of Jordan type ${\mathbf 2}\oplus{\mathbf 2}$, while the Killing operator is of type ${\mathbf 2}$. 

Interestingly, if $c=1$ and $b=8a/3$ then this is Ricci-flat. 
\subsubsection{Case [1,1]}
Universality order: 3. Ricci and Killing operators of Jordan type: ${\mathbf 2}\oplus{\mathbf 2}$.

An example in 5 dimensions is the solvable Lie algebra $\mathfrak{s}_{5,35}$ with structure coefficients: 
\[ C^4_{~41}=a+2,\quad C^2_{~21}=a+1,\quad C^5_{~51}=a,\quad C^3_{~31}=C^4_{~23}=C^2_{~53}=1.\]
The Riemann tensor is non-zero and Ricci tensor and Killing form are: 
\[ \Ric=\frac 12(a^2+4a-3)\theta^1\theta^1+(2a+3)\theta^3\theta^3, \qquad B=3(a^2+2a+2)\theta^1\theta^1.\] 
This example has Ricci operator of Jordan type  ${\mathbf 2}\oplus{\mathbf 2}$, while the Killing operator is of the simpler Jordan type  ${\mathbf 2}$.

\subsubsection{Case [1,$\tfrac{\mathbf 1}{\mathbf 2}$]}
Universality order: 3. Ricci and Killing operators of Jordan type: ${\mathbf 2}$.

A class of examples has been found for dimension $\geq$ 5. Consider the following solvable Lie algebra in dimension $\geq 5$ of signature $(2,2+k)$. 
\beq
&&C^2_{~12}=2, ~~C^3_{~13}=1, ~~C^4_{~14}=1, ~~C^2_{~34}=a, ~~C^2_{~13}=b, ~~C^4_{~13}=c, ~~C^2_{~14}=d, ~~C^5_{~13}=e, \nonumber \\
&&C^2_{~1i}=f_i, ~~C^4_{~1i}=g_i, ~~C^4_{~3i}=h_i, ~~C^2_{~3i}=ag_i-dh_i, \quad i=5,...,4+k. 
\eeq
This Lie algebra is (at least) decomposible into a direct sum of a 5-dimensional solvable Lie algebra, and the Abelian Lie algebra $\mathbb{R}^{k-1}$.  

Computing the Ricci tensor and the Killing form give: 
\[ \Ric=\frac 12(4-a^2)\theta^1\theta^1 \qquad B=6\theta^1\theta^1.\] 
Interestingly, if $a=\pm 2$, then the metric is Ricci-flat.

The perhaps simplest example in this class is the solvable Lie algebra $\mathfrak{s}_{5,36}$. A class of pseudo-Riemannian metrics can be given with structure coefficients ($a\neq 0$): 
\[ C^2_{~12}=2,\quad C^3_{~13}=1, \quad C^4_{~14}=1, \quad C^2_{~34}=a, \quad C^4_{~35}=1.\] 
The Riemann tensor is non-zero and Ricci tensor and Killing form are given above. This Lie algebra thus admits a Ricci-flat pseudo-Riemannian metric. 
\subsubsection{Case [1,0]}
Universality order: 3. Ricci and Killing operators of Jordan type: ${\mathbf 2}$.

Lie algebra is solvable, with a 2-step nilpotent nilradical. 

A class of 5-dimensional metrics of this type can be given.  Consider the following non-zero structure coefficients: 
\beq
C^2_{~12}, ~~ C^I_{~13}, ~~ C^I_{~14}, ~~C^I_{~15},  ~I=2,3,4,5,
\eeq
then this is a class of Lie algebras of case  [1,0].  Among these are also Ricci-flat metrics. 

Two super-simple examples of the case [1,0] are the solvable Lie algebras $\mathfrak{s}_{3,1 (a=-1)}\oplus \mathbb{R}^{k+1}$ and $\mathfrak{s}_{3,3 (\alpha=0)}\oplus \mathbb{R}^{k+1}$ with  structure coefficients: 
\[ C^4_{~13}=1, \quad C^3_{~14}=\pm 1\] 
The Riemann tensor is non-zero and Ricci tensor and Killing form for these are: 
\[ \Ric=\mp 2\theta^1\theta^1 \qquad B=\pm2\theta^1\theta^1.\] 
These two examples have $\nabla\Riem=0$ but the more general examples do not have this property.

\subsubsection{Case [$\tfrac{\mathbf 3}{\mathbf 4}$,$\tfrac{\mathbf 1}{\mathbf 2}$]}
Universality order: 2. Ricci and Killing operators are zero. 3-Symmetric. 

Lie algebra is 4-step nilpotent, but general type not allowed by the Bianchi identity. Lie algebras must belong to one of the simpler cases [$\tfrac{ 2}{ 3}$,$\tfrac 13$] or [$\tfrac1{2}$,$\tfrac12$] below. This case and its subcases below will be considered separately in sections \ref{RicciFlat} and \ref{Riemannzero} below. 

\subsubsection{Case [$\tfrac{\mathbf 2}{\mathbf 3}$,$\tfrac{\mathbf 1}{\mathbf 3}$]}
Universality order: 2. Ricci and Killing operators are zero. 1-Symmetric. 

Lie algebra is 4-step nilpotent. 

\subsubsection{Case [$\tfrac{\mathbf 1}{\mathbf 2}$,$\tfrac{\mathbf 1}{\mathbf 2}$]}
Universality order: 2. Ricci and Killing operators are zero. 1-Symmetric. 

Lie algebra is 3-step nilpotent. 

\subsubsection{Case [$\tfrac{\mathbf 2}{\mathbf 5}$,$\tfrac{\mathbf 1}{\mathbf 5}$]}
Universality order: 2. Ricci and Killing operators are zero. Riemann is zero.

Lie algebra is 3-step nilpotent. 

\subsubsection{Case [$\tfrac{\mathbf 1}{\mathbf 2}$,$\tfrac{\mathbf 1}{\mathbf 4}$]}
Universality order: 2. Ricci and Killing operators are zero. Riemann is zero. 

Lie algebra is 2-step nilpotent. 

\subsubsection{Case [$\tfrac{\mathbf 1}{\mathbf 2}$,0]}
Universality order: 2. Ricci and Killing operators are zero. Riemann is zero. 

Lie algebra is 2-step nilpotent. 

\subsubsection{Case [$\tfrac{\mathbf 1}{\mathbf 3}$,$\tfrac{\mathbf 1}{\mathbf 3}$]}
Universality order: 1. Ricci and Killing operators are zero. Riemann is zero. 

Lie algebra is 2-step nilpotent. 

\subsection{A class of Ricci flat metrics}
\label{RicciFlat} 
Let us now consider the Lie algebras with non-zero $C^a_{~bc}$ for only $3b_1+2b_1\leq -4$, i.e., Case  [$\tfrac3{4}$,$\tfrac12$], and its subcases. 
This implies the only possible non-zero components: 
\beq
&&(-2,1):~~C^2_{~14}, \qquad (-2,0): ~~C^2_{~1i}, \qquad (-2,-1):~~C^2_{13} \nonumber \\
&& (-1,-1): ~~C^i_{~13},~C^4_{~1i}, ~C^2_{~3i}, \qquad (-1,-2):~~C^4_{~13}, \qquad (0,-2):~~C^4_{~3i}.
\eeq
We notice that the Ricci tensor and Killing form are necessarily zero, hence it is Ricci flat. 

More specifically, $[\g,\g]:=\g_1\subset{\rm span}\{e_2,e_4,e_i\}$, and $\g_1$ is Abelian. 
Furthermore, $[\g,\g_1]:=\g_2\subset{\rm span}\{e_2,e_4\}$, and $[\g,\g_2]=\g_3\subset{\rm span}\{e_2\}$ and consequently, $\g$ is 4-step nilpotent. We note however, that if $C^2_{~41}=0$, then $\g_3=0$ and it is 3-step nilpotent. 

Checking the Jacobi identity,   the only components not being identically zero are: 
\[ (-2,-1): \quad C^2_{~14}C^4_{~3i} = 0.\]
Hence, we get: 
\[ C^2_{~14}=0 \quad \vee \quad C^4_{~3i} = 0 \] 

Computing the Riemann tensor we notice that this it not necessarily zero. However, the 3-step and 4-step cases give us: 
\begin{enumerate}
\item{} Case  [$\tfrac3{4}$,$\tfrac12$]: If $C^2_{~41}\neq 0$, and $C^4_{~3i}\neq 0$ then $\nabla^{(3)}\Riem=0$. As explained above, the general case is not allowed and Lie algebras have to be one of the subcases below. 
\item{} Case  [$\tfrac2{3}$,$\tfrac13$]:  If $C^2_{~41}\neq 0$, and $C^4_{~3i}=0$, then $\nabla\Riem=0$. 
\item{} Case  [$\tfrac1{2}$,$\tfrac12$]: If $C^2_{~41}=0$, then $\nabla\Riem=0$. \end{enumerate}
As an example of case [$\tfrac1{2}$,$\tfrac12$] take the following nilpotent Lie algebra called $\mathfrak{n}_{5,5}$ with left-invariant 1-forms: 
\beq
&&\theta^1=\d x, \qquad \theta^2=\d w+x\d z, \qquad \theta^3=\d y, \nonumber \\
&&\theta^4=\d v+y\d z+\tfrac 12y^2\d x, \qquad \theta^5=\d z+y\d x,
\eeq
with metric
\[ g=2\theta^1\theta^2+2\theta^3\theta^4+(\theta^5)^2.\]
Computing the Ricci, Riemann and $\nabla$Riemann tensor: 
\beq
\Ric=0, \qquad \Riem\neq 0, \qquad \nabla\Riem=0. 
\eeq
Indeed, in his Master's thesis, Markestad \cite{Markestad} has shown that all nilpotent indecomposable Lie algebras $\mathfrak{n}_{5,1}$-$\mathfrak{n}_{5,6}$ allow for a left-invariant metric of type [$\tfrac1{2}$,$\tfrac12$]. 

An example for the case [$\tfrac2{3}$,$\tfrac13$], we can take nilpotent algebra $\mathfrak{n}_{5,6}$ with the non-zero structure coefficients: 
\[ C^2_{~14}=C^5_{~13}=C^4_{15}=C^2_{~35}=1.\] 

\subsection{The Riemann flat cases}
\label{Riemannzero}
In this section we will consider the Lie algebras allowing for a flat metric, i.e., Riemann tensor is zero. 

The case [$\frac 25, \frac 15$], corresponds to a 3-step nilpotent Lie algebra. Up to dimension 4, there is only one such Lie algebra being properly 3-step nilpotent, namely $\mathfrak{n}_{4,1}$. Indeed, this provides with a 4-dimensional example of a flat metric: Choose
\[ C^2_{~14}=1 \quad C^4_{~13}=1, \]
we get a flat metric on $\mathfrak{n}_{4,1}$ of case [$\frac 25, \frac 15$]. 

For the case [$\frac 12, \frac 14$], is a 2-step nilpotent Lie algebra. Here, there are several 2-step nilpotent algebras in dimension 5. As an example, choose $\mathfrak{n}_{5,1}$ and the non-zero structure coefficients: 
\[ C^2_{~15}=1, \qquad C^4_{~13}=1. \] 
This gives an example of this case. 

Case  [$\frac 12, 0$], again a 2-step nilpotent Lie algebra, we can get an example by considering $\mathfrak{n}_{3,1}\oplus \mathbb{R}^{k+1}$. All Lie algebras with only non-zero structure coefficients being
\[ C^2_{~14}, ~C^2_{~1i},~C^2_{13}\] 
are examples of this Lie algebra. 

For the case [$\frac 13, \frac 13$], is a 2-step nilpotent Lie algebra, we can again get an example by considering $\mathfrak{n}_{3,1}\oplus\mathbb{R}^{k+1}$. This time with non-zero structure coefficients being: 
\[ C^4_{~13}, ~C^2_{~13}.\] 

\section{Signature $(3,3+k)$.} 
As the dimension of the null-directions increase the number of cases increases drastically. In signature $(3,3+k)$ we can represent each case with a triple [$x,y,z$] so that the components of the structure constants obey: 
\[ xb_1+yb_2+zb_3\leq -1.\] 
To list all cases would be too lengthy, but let us give some examples for some of the cases. Here, we assume the left-invariant frame is : $\{e_1,e_2,e_3,e_4, e_5,e_6,e_i\}$. Corresponding left-invariant 1-forms: $\{\theta^1,\theta^2,\theta^3,\theta^4, \theta^5,\theta^6,\theta^i\}$. 

Metric: 
\beq \label{metric33}
g=2\theta^1\theta^2+2\theta^3\theta^4+2\theta^5\theta^6+\sum^{6+k}_{i=7}(\theta^i)^2.\eeq
\subsection{Case $[3,1,1]$} 
\paragraph{Example $(\mathfrak{s}\mathfrak{l}(2,\mathbb{R})\oplusrhrim \mathbb{R}^2)\oplus\mathbb{R}^{k+1}$: }
An interesting example can now be found by extending the non-solvable example from case [3,1] to a non-trivial Levi decomposition. This give rise to an example given on the semi-direct sum: $(\mathfrak{s}\mathfrak{l}(2,\mathbb{R})\oplusrhrim \mathbb{R}^2)\oplus\mathbb{R}^{k+1}$. 
We let $\mathfrak{s}\mathfrak{l}(2,\mathbb{R})=\span\{e_1,e_3,e_4\}$ and $\mathbb{R}^2=\span\{e_2,e_6\}$, and so the representation of $\sl$ acts as a derivation on $\mathbb{R}^2=\span\{e_2,e_6\}$. For example, we can choose the following non-zero structure constants: 
\beq
&&(-1,2,0):~~ C^3_{~14}=1, \qquad (-1,1,1):~~ C^2_{~46}=1, \qquad (0,0,-1): ~~C^6_{~12}=-1, \nonumber \\
&&(0,-1,0):~~C^1_{~13}=2, ~~C^4_{~43}=-2, ~~C^2_{~32}=1, ~~C^6_{~36}=-1.
\eeq
Computing the Ricci tensor and the Killing form we get: 
\[ \Ric=6\theta^1\theta^4+2\theta^1\theta^6-5\theta^3\theta^3-\frac 52\theta^5\theta^5, \quad B=-10(\theta^1\theta^4-\theta^3\theta^3).\]
Here, the Ricci tensor is of Jordan type : ${\mathbf 4}\oplus {\mathbf 2}$, while the Killing form is of type  ${\mathbf 4}$.

\subsection{Case $[1,1,1]$} 
There are interesting examples of case $[1,1,1]$ also involving Lie algebras with non-trivial Levi decompositions. The following examples come from Markestad's Master thesis \cite{Markestad}. 

We consider the following completely null subspaces: 
\[ N_-=\span\{e_1,e_3,e_5\}, \qquad N_+=\span\{e_2,e_4,e_6\}.\]
 Let $\mathfrak{h}$ be one of the two 3-dimensional simple Lie algebras, and consider the semi-direct sum $\mathfrak{h}\oplusrhrim \mathbb{R}^3$ where we identify:
\[  N_-=\mathfrak{h},\qquad N_+=\mathbb{R}^3\] 
The semi-direct sum is defined by a representation $\mathsf{D}$ of $\mathfrak{h}$ acting for each $X\in \mathfrak{h}$ on $\mathbb{R}^3$ as a derivation, i.e.:
\[ \mathsf{D}_X:\mathbb{R}^3\rightarrow\mathbb{R}^3,\] 
is a derivation of $\mathbb{R}^3$. 
There are two cases to consider: 
\begin{enumerate} 
\item{} The representation $\mathsf{D}$ is trivial, i.e., it acts as the zero map. 
\item{} The representation $\mathsf{D}$ is non-trivial (in total three different non-isomorphic Lie algebras). 
\end{enumerate}
The first case has only non-trivial structure constants of the form: $C^-_{~--}$. Here a minus implies an odd index (i.e. in $N_-$). By lowering, this gives $C_{+--}$. Hence, these must obey $b_1+b_2+b_3=-1$, as required. 
By computing the Ricci tensor we get: 
\[ \Ric=-\frac 12B.\] 
The Ricci tensor is therefore degenerate on $N_-$, and of type ${\mathbf 2}\oplus{\mathbf 2}\oplus{\mathbf 2}$

For the second case when $\mathsf{D}$ is non-trivial, we get the additional non-zero structure coefficients given on the form: $C^+_{~-+}$. By lowering we get $C_{--+}$, again giving $b_1+b_2+b_3=-1$, as required. So this construction gives examples of metrics of case [1,1,1].

\paragraph{Higher dimensions} We note that this construction allows for generalisation to higher dimensions. Assume that the pseudo-Riemannian space is of signature $(p,p+k)$. Then we have two totally null subspaces: 
\[ N_-=\span\{e_1,e_3,..., e_{2p-1}\}, \qquad N_+=\span\{e_2,e_4,...,e_{2p}\}.\] 
Consider the semi-direct sum Lie algebra $\mathfrak{h}\oplusrhrim \mathbb{R}^{p+k}$, and let 
\[  N_-= \mathfrak{h},\qquad N_+=\mathbb{R}^p.\] 
This construction would thus give examples of case [1,1,...,1] in higher dimensions as well. 

\subsection{Generalisations with a non-abelian radical}  There are further generalisations pseudo-Riemannian metrics on non-trivial Levi decomposable Lie algebras. For example, in 6 dimensions we can use $\mathfrak{s}\mathfrak{l}(2,\mathbb{R})\oplusrhrim \mathfrak{n}_{3,1}$ with structure constants given the same as in eq.(\ref{sl2Rn}), $a\neq 0$, $b=1$, but with metric eq.(\ref{metric33}). This example thus provides us with an example of an indecomposable Lie algebra with non-abelian radical. 

Furthermore, we have also found an example of class [3,1,1] with the Levi decomposition  $\mathfrak{s}\mathfrak{l}(2,\mathbb{R})\oplusrhrim \mathfrak{s}_{3,1, a=1}$ in 6 dimensions with neutral signature.  Given the structure constants: 
\beq
(0,-1,0):&&~~C^6_{~36}=-2, \quad C^1_{~13}=-2, \quad C^4_{~34}=C^2_{~23}=C^4_{~12}=1, \nonumber \\
(-1,1,1):&& ~~C^3_{~16}=1, \quad C^2_{~46}=-1, \nonumber \\
(0,0,-1): && ~~C^2_{~25}=\alpha, \quad C^4_{~45}=\alpha, \quad \alpha\neq 0,
\eeq
with metric 6-dimensional neutral metric
\[ g=2\theta^1\theta^2+2\theta^3\theta^4+2\theta^5\theta^6.\]
Then the Ricci tensors and Killing form are 
\[ \Ric =-2\theta^1\theta^6-4\theta^3\theta^3+4\alpha\theta^3\theta^5-\alpha^2\theta^5\theta^5, \quad B=10(\theta^1\theta^6+\theta^3\theta^3)+2\alpha^2\theta^5\theta^5.\]

There are clearly more examples of this kind in higher dimensions. On the other hand, there are clear restrictions for their existence, but will be left for further study.

\section{Pseudo-Riemannian nil- and solvmanifolds} 
Consider a solvable (possibly nilpotent) Lie group, with Lie algebra $\g$, equipped with the pseudo-Riemannian metric
\[ g=2\sum_{i=1}^p\theta^{2i-1}\theta^{2i}+\sum_{i=2p+1}^{2p+k}(\theta^i)^2.\] 
Define the two totally null subspaces: 
\beq
N_-=\span\{e_1,e_3,e_5,...,e_{2p-1}\}, \quad N_+=\span\{e_2,e_4,e_6,...,e_{2p}\},
\eeq
and the space-like subspace:
\beq
H=\span\{e_{2p+1},...,e_{2p+k}\}.
\eeq
We can similarly identify the basis ${e_\mu}$ as a basis for $\g$, and $\theta^\mu$ as the corresponding left-invariant one-forms. 

Since the Lie algebra is solvable, the derived series: 
\[ \g_{0}:=\g. \quad \g_{i+1}:=[\g_i,\g_i] \] 
will terminate; i.e., there exists a $K$ so that $\g_K=0$. Define also the subspaces: 
\[ V_1:=\g\slash \g_1, \quad V_i:=\g_{i-1}\slash \g_i, ~i=2,...,K-1, \quad V_K:=\g_{K-1}. \]
We can now arrange each of these subspaces and identify them with the span of basis vectors $e_\mu$ in the following order: 
\beq
\{e_1,e_3,e_5,...,e_{2p-1}, e_{2p+1},...,e_{2p+k}, e_{2p},..., e_6,e_4,e_{2}\}\sim
V_1\oplus V_2\oplus ...\oplus V_{K-1}\oplus V_K.
\eeq
For example, if all $V_i$ are of dimension 2, then $V_1=\span\{e_1,e_3\}$, $V_2=\span\{e_5,e_7\}$, ..., $V_K=\span\{e_4,e_2\}$. 

We will make the further assumption  that: 
\beq
[H,H]\subset N_+. 
\eeq
Then, we can make a refinement of the spaces $\{V_i\}$  to $\{W^-_i,H,W^+_i\}$ in the following way: 
\beq
\{e_1,e_3,e_5,...,e_{2p-1}, e_{2p+1},...,e_{2p+k}, e_{2p},..., e_6,e_4,e_{2}\}\sim\{ W^-_1,W^-_2, ..., W^-_{\tilde{p}}, H, W^+_{\tilde p},..., W^+_2, W^+_1\} \nonumber 
\eeq
where  subspaces $W^\pm_i$ are null, $\dim(W_i^+)=\dim(W_i^-)$,  and each pair $(W^\pm_i,W^\pm_j)$, $i\neq j$ is mutually orthogonal. %except for $W^-_i$ and $W^+_i$. 

Under the following additional assumptions: 
\beq\label{eqW-W+}
\left[W^-_i,W^+_j\right]&&\subset
\begin{cases} 
 \span\{W^+_{j},...,W^+_1\}, & j\leq i, \\
 \span\{ W^-_{i+1},..,W^-_{\tilde{p}}, H,W^+_{\tilde{p}},...,W^+_1\} & j>i,
\end{cases} 
\\      
\left[H,W^+_j\right] &&
\subset \span\{W^+_{j-1},...,W^+_1\},\\
\label{eqW+W+}
\left[W^+_i,W^+_j\right]_{i\geq j} &&
\subset \span\{W^+_{j-1},...,W^+_1\},
\eeq
the Lie algebra $\mathfrak{g}$ is in the null cone $\mathcal{N}$. 

To show this, we define the coefficients, $x_i$, of the equation 
\[ x_1b_1+x_2b_2+...+x_{p}b_{p}\leq -1,\]
where each non-zero structure coefficient has boost weight $(b_1,...b_i,...,b_p)$. 
For all $b_k \in W^{\pm}_k$, the $x_k$'s are equal. Then, since the structure coefficients belong to a 3-index object, we define the weight of $W^-_{i}$ to be larger than the sum of any two weights $x_j$ with larger $j>i$, i.e., $x_i> x_j+x_k$, $i<j,k$. 

We can achieve this by choosing: $x_{\tilde{p}}=1$, then iteratively $x_{i-1}=2x_i+1$ (giving $x_i=2^{\tilde{p}+1-i}-1$). 
\subsection{All nilpotent Lie algebras are in the null cone of the split action.}
Let consider the case when the pseudo-Riemannian space is a nilmanifold, i.e., the Lie algebra is nilpotent. Then:
\begin{prop}\label{PropNil}
Given a nilpotent Lie algebra, $\mathfrak{n}$, then there exists a pseudo-Riemannian metric of signature $(p,p)$ if $\dim(\mathfrak{n})$ is even, and $(p,p+1)$ if $\dim(\mathfrak{n})$ is odd, on the corresponding nilmanifold so that $\mu$ is in the null cone $\mathcal{N}$.
\end{prop}
\begin{proof}
In this case we can define the sequence of subalgebras $\g_{0}:={\mathfrak n}$, and  $\g_{i+1}:=[\g_0,\g_i] $. Define then $V_i:=\g_{i-1}/\g_i$, and the refinements  $\{W^-_i,H,W^+_i\}$ as before. For the signatures in question $H$ is either of dimension 0, or 1, and consequently $[H,H]=0$. Similarly, due to the fact that $\g_{i+1}:=[{\mathfrak n},\g_i] $ all the other conditions eqs. (\ref{eqW-W+})-(\ref{eqW+W+}) are satisfied as well. 
\end{proof}

\subsection{Some solvable examples} 
Consider the following class of solvable extensions of nilpotent Lie algebras: 
Let $\mathfrak{s}=\mathfrak{a}\oplus\mathfrak{n}$, where $\mathfrak{a}$ is an abelian subalgebra, and $\mathfrak{n}$ is the nilradical. Then: 
\[ [\mathfrak{a},\mathfrak{a}]=0, \quad [ \mathfrak{a},\mathfrak{n}]\subset \mathfrak{n}, \quad  [ \mathfrak{n},\mathfrak{n}]\subset \mathfrak{n}.\] 
Furthermore, defining: 
\[  \mathfrak{g}_1:=[ \mathfrak{s}, \mathfrak{s}]\subset \mathfrak{n}, \qquad \mathfrak{g}_{i+1}:=  [ \mathfrak{g}_1, \mathfrak{g}_i]\subset\mathfrak{g}_i,   \] 
implies that $\mathfrak{g}_{i}$ terminates, and each $\mathfrak{g}_{i+1}$ is a proper subset of $\mathfrak{g}_{i}$. Let us also assume that the solvable Lie algebra has the following property:
\[  [ \mathfrak{a},\mathfrak{g}_i]\subset \mathfrak{g}_{i+1}.\]

By the above assumptions, we get the following result: 

\begin{prop}
Given a solvable extension, $\mathfrak{s}$,  of a nilpotent Lie algebra, with the assumptions given above, then there exists a pseudo-Riemannian metric of signature $(p,p)$ if $\dim(\mathfrak{s})$ is even, and $(p,p+1)$ if $\dim(\mathfrak{s})$ is odd, on the corresponding Lie group so that $\mu$ is in the null cone $\mathcal{N}$.
\end{prop}
\begin{proof}\label{PropSplit}
We note that the series  defined above, $\mathfrak{g}_i$ the spaces $V_i:=\mathfrak{g}_{i-1}/\mathfrak{g}_i$ implies that $[V_i,V_j]_{i\leq j}\subset \span\{V_{j+1},...,V_K\} $. Moreover, since $H$ must be of dimension 1 or 0, $[H,H]=0$. The corresponding refinement \newline $\{ W^-_1,W^-_2, ..., W^-_{\tilde{p}}, H, W^+_{\tilde p},..., W^+_2, W^+_1\}$ will thus fullfill the requirements above. 
\end{proof}

This result gives a sufficient condition for the split signature case. However, many solvable algebras allow for other signatures as well. As an extreme example of this is when the nilpotent group is actually abelian. Then the Lorentzian signature suffices (and all other signatures as well). 

As an example in the other end of the spectrum, consider a filiform nilradical $\mathfrak{n}$ and a solvable extension $\mathfrak{s}=\mathfrak{a}\oplus\mathfrak{n}$  with  $\dim(\mathfrak{a})=1$ and $[\mathfrak{s},\mathfrak{s}]=\mathfrak{n}$. Then each $V_{i}=\mathfrak{g}_{i-1}/\mathfrak{g}_{i}$  has dimension 1, except for $\dim(V_{2})=2$. These examples provide candidates for examples the signature in the proposition is sharp and 
the obvious candidates to saturate the equation: 
\[ \left(2^p-1\right)b_1+...+\left(2^{p+1-i}-1\right)b_i+...+3b_{p-1}+b_p\leq -1. \]  

\subsection{Are all solvable Lie algebras in the null cone of some $O(p,q)$ action?} 
The answer is no\footnote{This seems to be because we are considering Lie algebras and orbits over $\mathbb{R}$. Over $\mathbb{C}$ all solvable Lie algebras are completely solvable, and hence using the nullcone of $O(n,\mathbb{C})$, seem to fall under the spell of the Prop. \ref{PropCSol} below.}. Before we give examples which are not let us investigate this question a bit more carefully and give a sufficient condition for when a solvable Lie algebra is in the null cone. 
Consider a solvable Lie algebra $\mathfrak{s}$. Then by Prop. 1.23 in \cite{Knapp} there will exist a sequence of subalgebras
\[ \mathfrak{s}=\mathfrak{a}_0\supset \mathfrak{a}_1\supset \dots\mathfrak{a}_i\supset \dots \mathfrak{a}_n=0,\]
such that each $i$, $\mathfrak{a}_{i+1}$ is an ideal of $\mathfrak{a}_i$, and $\dim(\mathfrak{a}_{i+1}/\mathfrak{a}_i)=1$. 
This is called an \emph{elementary sequence} and to such a sequence we can clearly try to apply the construction given in the proof of Prop.\ref{PropSplit}. On the other hand, what seems to be critical is the requirement 
\beq\
\left[W^-_i,W^+_j\right]\subset 
 \span\{W^+_{j},...,W^+_1\}, \quad j\leq i.
\eeq

On the other hand, for \emph{completely solvable} Lie algebras \cite{Knapp}, there exists an elementary sequence of ideals and for such Lie algebra, the above requirement can be fulfilled. Hence: 
\begin{prop}\label{PropCSol}
Given a completely solvable Lie algebra, $\mathfrak{s}$, then there exists a pseudo-Riemannian metric of signature $(p,p)$ if $\dim(\mathfrak{s})$ is even, and $(p,p+1)$ if $\dim(\mathfrak{s})$ is odd, on the corresponding Lie group so that $\mu$ is in the null cone $\mathcal{N}$.
\end{prop}

To find solvable Lie algebras which \emph{do not} lie in the null cone it is sufficient to consider dimension 3 with Lorentzian signature (1,2). Consider the Lie algebras of Euclidean motions of the plane ($\alpha =0$), or a subalgebra of the homothetic motions ($\alpha>0$) of the plane: 
\beq
{\mathfrak s}_{3,3}:\quad [e_2,e_1]=\alpha e_2-e_3, \qquad  [e_3,e_1]=e_2+\alpha e_3.
\eeq
These three-dimensional Lie algebras do not lie in any null cone. To see this we note that $\g_1:=[\g,\g]=\span\{e_2,e_3\}$. Consider the endomorphism ${\sf A}: \g_1\rightarrow\g_1$ defined by ${\sf A}(X):=[e_1,X]$. According to section \ref{sect:typeIII}, ${\sf A}$ necessarily has 2 real eigenvalues. However, for ${\mathfrak s}_{3,3}$ the eigenvalues are necessarily \emph{complex}. Since the Lorentzian case is the only possibility for non-trivial null cone in dimension 3, ${\mathfrak s}_{3,3}$ cannot be in any null cone. 

In higher dimensions such examples not lying in the null cone are more difficult to find, and, in particular, the author has not been able to find solvable examples of this kind in even dimensions. 
So the question remains: \\
\begin{q}
Given a solvable Lie algebra, $\mathfrak{s}$. When does there exist a pseudo-Riemannian metric of signature $(p,p+k)$ on the corresponding solvmanifold so that $\mu$ is in the null cone $\mathcal{N}$?
\end{q}

\section{The structure of the pseudo-Riemannian manifolds}
It is clear that all the left-invariant metrics on the Lie groups in the null cone has a certain degenerate structure. Consider the general metric in signature $(p,p+k)$ in the null cone. Then define the two totally null subspaces: 
\beq
N_-=\span\{e_1,e_3,e_5,...,e_{2p-1}\}, \quad N_+=\span\{e_2,e_4,e_6,...,e_{2p}\},
\eeq
and the space-like subspace
\beq
H=\span\{e_{2p+1},...,e_{2p+k}\},
\eeq
as before. Then the null-distribution $\mathcal{D}:=N_+$ as well as its orthogonal complement $\mathcal{D}^{\bot}:=N_+\oplus H$, are \emph{integrable}. This can be seen since: 
\[ [N_+,N_+]\subset N_+, \qquad [N_+,H]\subset N_+, \quad [H,H]\subset N_+.\] 
This can be seen since the structure constants $C^a_{~++}=C^a_{~i+}=C^a_{~ij}=0$, for $a=\{-, i\}$ and $i,j=2p+1,...,2p+k$,  fullfill
\beq \label{ineq}
 x_1b_1+x_2b_2+...+x_pb_p\leq -1.
 \eeq
Indeed, from the same vanishing structure constants we also note we have a codimension one sequence, $\mathcal{D}_i$, of integrable distributions, so that
\[ \mathcal{D}\subset \mathcal{D}_1\subset...\subset \mathcal{D}_{k-1}\subset\mathcal{D}_k=\mathcal{D}^{\bot}. \] 
This implies that we can introduce coordinates at each level related to the integrable structure. Indeed, by checking the condition $ [N_+,N_+]\subset N_+$ and requiring eq.(\ref{ineq}) where the $b_i$ are reversely well-ordered, then we get that 
\[ C^k_{~ij}\neq 0 \quad \Rightarrow k<i,j,\]
where $i,j,k=2,4,...,2p$.  This implies that there is a descending sequence: 
\[ \mathcal{D}\supset\mathcal{D}_{-1}\supset...\supset\mathcal{D}_{-\tilde{p}}=0.  \] 
From a Lie algebra perspective, this implies that the complement $\mathcal{D}^{\bot}$ forms a \emph{nilpotent subalgebra}. Hence, this can be extended to the following result: 
\begin{prop}
Assume that the Lie algebra $\mu$, is in the null cone $\N$. Then there exists a pseudo-Riemannian metric so that the distribution $\mathcal{D}^{\bot}:=N_+\oplus H$ is a nilpotent Lie subalgebra. Consequently, there exists a codimension 1 sequence of integrable distributions: 
\[0=\mathcal{D}_{-p}\subset \mathcal{D}_{-p+1}\subset ...\subset\mathcal{D}_{-1}\subset\mathcal{D}\subset \mathcal{D}_1\subset ...\subset \mathcal{D}_k=\mathcal{D}^\bot.\]
\end{prop}
As an immediate corollary, we get the following: 
\begin{cor}\label{CorNil}
Assume that a $(2p+k)$-dimensional Lie algebra $\g$ is in the null cone under an $O(p,p+k)$-action, as described here. Then  $\g$ has a nilpotent subalgebra of dimension $(p+k)$. 
\end{cor} 
%\section{Conclusion and Outlook} 

\appendix
\section{Structure constants for signature $(2,2+k)$.} 
\label{Cabc}
In the null cone, the following structure constants may not be zero: 
\beq
&&(-1,2): ~~C^3_{~14}, \qquad (-1,1): ~~C^2_{~4i}, ~C^3_{~1i},~C^i_{~14} \qquad (-2,1):~~C^2_{~41}, \qquad (-2,0): ~~C^2_{~1i}, \nonumber \\ 
&& (-1,0): ~~C^2_{~ij},~C^i_{~1j},~ C^2_{~12},~C^2_{~34}, ~C^3_{~13}, ~C^4_{~14}, \qquad (-2,-1):~~C^2_{13} \qquad \nonumber \\
&& (-1,-1): ~~C^i_{~13},~C^4_{~1i}, ~C^2_{~3i},\qquad (0,-1): ~~C^4_{~ij},~C^i_{~3j}, ~C^4_{~34}, ~C^4_{~12},~C^1_{~13},~C^2_{~23},  \nonumber \\
&& (-1,-2):~~C^4_{~13}, \qquad (0,-2):~~C^4_{~3i}, \qquad (1,-2):~~ C^4_{~23}.
\eeq

\section{Two $SL(3,\mathbb{R})$ examples} 
Here we will consider two examples having a  $SL(3,\mathbb{R})$ subgroup, i.e., it involves the subalgebra $\mathfrak{s}\mathfrak{k}(3,\mathbb{R})$. The examples are of dimensions 10 and 12. Due to the enumeration ambiguity, we will use {\rm hexadecimals} for the indicies; i.e., $A=10$, $B=11$, etc. 
They are of class [5,3,3,1,1] and [5,3,3,1,1,1].
\subsection{An $\mathfrak{s}\mathfrak{k}(3,\mathbb{R})\oplus\mathbb{R}^2$ example in signature $(5,5)$}
Let 
\[ \mathfrak{s}\mathfrak{k}(3,\mathbb{R})=\span\{e_1,e_3,e_5,e_6,e_7,e_8,e_9, e_A\},\qquad \mathbb{R}^2=\span\{e_2,e_4\},\] 
with structure constants: 
\beq\label{sl3RxR2}
&&C^1_{~17}=-2, \quad C^6_{~67}=2, \quad C^7_{~16}=1 \quad C^3_{~37}=C^5_{~57}=C^8_{~78}=C^A_{~7A}=-1\nonumber \\
&& C^3_{~18}
=1, \quad C^5_{~1A}=C^8_{~36}=C^A_{~56}=-1, \quad C^3_{~39}=C^8_{~89}=C^A_{~9A}=-3, \quad C^5_{~59}=3,\nonumber \\
&& C^7_{~3A}=C^9_{~3A}=\frac 12, \quad C^1_{~35}=C^6_{~8A}=1, \quad C^7_{~58}=\frac 12,\quad C^9_{~58}=-\frac 12.
\eeq
Here, the set $\{e_2,e_4,e_6,e_8,e_A\}$ spans the nilpotent Lie algebra ${\mathfrak{n}}_{3,1}\oplus \mathbb{R}^2$.
\subsection{An $\left(\mathfrak{s}\mathfrak{k}(3,\mathbb{R})\oplusrhrim\mathbb{R}^3\right)\oplus\mathbb{R}$ example in signature $(6,6)$}
Let 
\[ \mathfrak{s}\mathfrak{k}(3,\mathbb{R})=\span\{e_1,e_3,e_5,e_6,e_7,e_8,e_9, e_A\},\qquad \oplusrhrim\mathbb{R}^3=\span\{e_2,e_4, e_C\}, \qquad \oplus\mathbb{R}=\span\{e_B\}.\] 
In addition to the non-zero structure constants given in eq.(\ref{sl3RxR2}), we have the additional non-zero structure constants: 
\beq
&&C^C_{~7C}=C^C_{~9C}=C^2_{~6C}=C^4_{~AC}=C^2_{~27}=C^C_{~12}=C^C_{~34}=1, \nonumber \\
  && C^2_{~29}=C^4_{~25}=C^2_{~48}=-1, \qquad C^4_{~49}=2.
\eeq
Here,  the set $\{e_2,e_4,e_6,e_8,e_A, e_B, e_C\}$ spans the nilpotent Lie algebra $\left({\mathfrak{n}}_{3,1}\oplusrhrim\mathbb{R}^3\right)\oplus\mathbb{R}$.

\section*{Acknowledgments}
The author is grateful to J. Lauret for discussions. 
%\acknowledgments

\renewcommand{\thesection}{\Alph{section}}
\setcounter{section}{0}

\renewcommand{\theequation}{{\thesection}\arabic{equation}}

\bibliographystyle{JHEP}

%%\bibliographystyle{my_cqg}
%
%%\bibliographystyle{elsart-num_mio}
%\bibliography{bibl}

%\bibliographystyle{unsrt}

\bibliography{bibl}

\end{document}